\documentclass[12pt]{article}
\usepackage{amssymb,amsmath,amsfonts,amsthm}
\usepackage{adjustbox,tabu}
\usepackage{enumerate}
\usepackage{textcomp}
\usepackage{hyperref}
\newcounter{parag}

\newtheorem{lem}{Lemma}[section]
\newtheorem{theorem}{Theorem}[section]

\newtheorem{prop}{Proposition}
\newtheorem{question}{Question}
\newtheorem{cor}{Corollary}
\newtheorem{rmk}{Remark}

\title{Axial view on pseudo-composition algebras and train algebras of rank 3}
\author{Ilya Gorshkov, Andrey Mamontov and Alexey Staroletov\footnote{
The first and second authors are supported by Russian Scientific Foundation (project \textnumero22-11-00081, https://rscf.ru/en/project/22-11-00081/)
and the third author is supported by RAS Fundamental Research Program (project  FWNF-2022-0002)}}

\date{\vspace{-30px}}

\begin{document}
\maketitle
\newcommand{\Addresses}{{
		\bigskip\noindent
		\footnotesize
	    Ilya~Gorshkov, \textsc{Sobolev Institute of Mathematics, Novosibirsk, Russia;}\\\nopagebreak
        \textsc{Saint Petersburg State University, Saint Petersburg, Russia;}\\\nopagebreak
		\textit{E-mail address: }
        \texttt{ilygor8@gmail.com}

        \medskip\noindent
        Andrey~Mamontov, \textsc{Sobolev Institute of Mathematics, Novosibirsk, Russia;}\\\nopagebreak
		\textsc{Saint Petersburg State University, Saint Petersburg, Russia;}\\\nopagebreak
		\textit{E-mail address: } \texttt{andreysmamontov@gmail.com}

		\medskip\noindent
		Alexey~Staroletov, \textsc{Sobolev Institute of Mathematics, Novosibirsk, Russia;}\\\nopagebreak
		\textit{E-mail address: } \texttt{staroletov@math.nsc.ru}

		\medskip
}}

\begin{abstract}
We show that pseudo-composition algebras and train algebras of rank~3 generated by idempotents are characterized as axial algebras with fusion laws derived from the Peirce decompositions of idempotents in these classes of algebras.

The corresponding axial algebras are called $\mathcal{PC}(\eta)$-axial algebras, where $\eta$ is an element of the ground field. As a first step towards their classification, we describe $2-$ and $3$-generated subalgebras of such algebras.
\end{abstract}


{\bf MSC classes:} 17A99, 17C27, 17D92

\section{Introduction}
A commutative algebra is said to be of rank 3 if every its element generates a subalgebra of dimension not greater than two.
This class of algebras includes many interesting examples, some general theory and motivation can be found in~\cite{W99}.
We are interested in two particular cases.

Suppose that $\mathbb{F}$ is a field of characteristic not 2 or 3.
A commutative $\mathbb{F}$-algebra $A$ endowed with
a non-zero symmetric bilinear form $\varphi$
is called a {\it pseudo-composition algebra} if $x^3=\varphi(x,x)x$ for all $x\in A$. These algebras have been actively studied in the past, in particular, Meyberg and Osborn obtained their classification, under some  restrictions, in~\cite{MO93}.
It was shown in~\cite{EO00,S59} that these algebras are closely related to Jordan algebras of generic degree $\leq 3$.
Moreover, it is known that
the bilinear form $\varphi$ is a Frobenius form, that is $(xy,z)=(x,yz)$ for all algebra elements $x$, $y$, and $z$~\cite{EO00}.

The second subclass are train algebras of rank~3.
Let $A$ be a commutative algebra over a field $\mathbb{F}$
with $\operatorname{char}\mathbb{F}\neq2,3$. The principal powers of $A$ are defined by $x^1=x$ and $x^i=x^{i-1}x$ for $i\geq2$. If there exists a non-zero algebra homomorphism $\omega:A\rightarrow\mathbb{F}$, then $A$ is called a {\it baric algebra} and $\omega$ is a {\it weight function}. In this case the pair $(A,\omega)$ is said to be a (principal) train algebra of rank~$r$, where $r$ is a positive integer, if there exist $\lambda_1,\ldots,\lambda_{r-1}\in\mathbb{F}$ such that every $x\in A$ satisfies the equality
$x^r+\lambda_1\omega(x)x^{r-1}+\ldots+\lambda_{r-1}\omega(x)^{r-1}x=0$.
These algebras were introduced by Etherington in 1939 as a part of the algebraic
formalism of genetics in his fundamental work~\cite{Eth39}.

Let us briefly discuss the concept of axial algebras.
A \emph{fusion law} over $\mathbb{F}$ is a pair $(\mathcal{F},\ast)$, where $\mathcal{F}$ is a subset of $\mathbb{F}$ and $\ast:\mathcal{F}\times\mathcal{F}\to 2^\mathcal{F}$ is a 
map to the set of all subsets of $\mathcal{F}$.
If $A$ is a commutative $\mathbb{F}$-algebra and $a\in A$, then $ad_a : A\to A$ stands for the adjoint map sending $u$ to $au$. For $\lambda\in\mathbb{F}$, denote $A_{\lambda}(a)=\{u\in A~|~au=\lambda u\}$ and for $L\subseteq\mathbb{F}$, denote $A_L(a):=\oplus_{\lambda\in L}A_{\lambda}(a)$.

An \emph{($\mathcal{F}$-)axial algebra} $A$ is a commutative algebra over $\mathbb{F}$ generated by a set of idempotents $X$, called \emph{axes}, such that for each $a \in X$ we have
$A=A_\mathcal{F}(a)$, and
$A_{\lambda}(a)A_\mu(a)\subseteq A_{\lambda\ast\mu}(a)$ for all
$\lambda,\mu\in\mathcal{F}$. An axis $a$ is said to be \emph{primitive} if $A_1(a)$ is $1$-dimensional; that is, $A_1(a)=\langle a\rangle$. An axial algebra $A$ is \emph{primitive} if $A$ is generated by a set of primitive axes.
These properties are reminiscent of the so-called Peirce decompositions in different classes of non-associative algebras. Nevertheless,
the main inspiration for the concept of axial algebras were the Griess algebra~\cite{Gr82} and Majorana theory~\cite{Iv09}. Axial algebras were introduced by Hall, Rehren, and Shpectorov in~\cite{HRS1,HRS2}. Current state-of-art of this topic can be found in a recent survey \cite{survey}.

In this paper we consider pseudo-composition algebras
and train algebras of rank~3 from the point of view of axial algebras. The two classes of algebras of rank 3, that we discussed above, are distinguished from others by the following property, which resembles axial behavior: if $a$ is an idempotent in $A$ and $ad_a$ is its adjoint operator, then eigenvalues of $ad_a$ lie in a fixed set of size 3 and eigenvectors obey a fusion law of size 3.

Finite fusion laws may be written in the form of tables.
There are three examples below in Table~\ref{t:1}.
Pseudo-composition algebras are described by $\mathcal{PC}(-1)$ \cite{Wal88}. The fusion law $\mathcal{J(\eta)}$ corresponds to axial algebras of Jordan type \cite{HRS2}, which include Matsuo algebras related to $3$-transposition groups ($\eta \not = 0,1)$ and Jordan algebras ($\eta=\frac{1}{2}$). The most
general case $\mathcal{J(\alpha,\beta)}$ is taken from \cite{Whybrow2}, where all possible 2-generated graded primitive axial algebras are listed.

\begin{table}[h]
\begin{center}
\renewcommand{\arraystretch}{1.4}
\begin{tabular}{|c||c|c|c|}
\hline
$\ast$&$1$&$\eta$&$\frac{1}{2}$\\
\hline\hline
$1$&$1$&$\eta$&$\frac{1}{2}$\\
\hline
$\eta$&$\eta$&$1$&$\frac{1}{2}$\\
\hline
$\frac{1}{2}$&$\frac{1}{2}$&$\frac{1}{2}$&$\eta,1$\\
\hline
\end{tabular} \hspace{5pt}
\begin{tabular}{|c||c|c|c|}
\hline
$\ast$&$1$&$0$&$\eta$\\
\hline\hline
$1$&$1$& &$\eta$\\
\hline
$0$&&$0$&$\eta$\\
\hline
$\eta$&$\eta$&$\eta$&$1,0$\\
\hline
\end{tabular} \hspace{5pt}
\begin{tabular}{|c||c|c|c|}
\hline
$\ast$&$1$&$\alpha$&$\beta$\\
\hline\hline
$1$&$1$&$\alpha$ &$\beta$\\
\hline
$\alpha$&$\alpha$&$1,\alpha$&$\beta$\\
\hline
$\beta$&$\beta$&$\beta$&$1,\alpha$\\
\hline
\end{tabular}
\end{center}
\caption{Fusion laws $\mathcal{PC}(\eta)$, $\mathcal{J(\eta)}$, and $\mathcal{J(\alpha,\beta)}$}\label{t:1}
\end{table}

It is noted \cite{Whybrow2} that the case $\beta=\frac{1}{2}$ is special in terms of multiplication formulas. Observe that
$\mathcal{J(\eta)}$ is a special case of $\mathcal{J}(0,\eta)$ where additionally $0*0=\{0\}$;
and there is a special subcase $\beta=\eta=\frac{1}{2}$ including Jordan algebras. Meanwhile
$\mathcal{PC}(\eta)$ is a special case of $\mathcal{J}(\eta,\frac{1}{2})$ where the product of $\eta$-eigenspaces is also further limited, but this time we exclude $\eta$, not~$1$. We show why the value $\eta=-1$
is special for $\mathcal{PC}(\eta)$ and prove that the fusion law characterizes pseudo-composition algebras in this case.

\begin{theorem}\label{th:1}
Suppose that $\mathbb{F}$ is a field of characteristic not $2$ or $3$. Suppose that $A$ is a primitive $\mathcal{PC}(\eta)$-axial algebra over $\mathbb{F}$, where $\eta\in\mathbb{F}$ and $\eta\not\in\{1, \frac{1}{2}\}$.
Then the following statements hold:
\begin{enumerate}[$(i)$]
    \item if $\eta=-1$, then $A$ is a pseudo-composition algebra;
    \item if $\eta\neq-1$, then $A$ is a train algebra of rank $3$.
\end{enumerate}
\end{theorem}
The converse statements are true under the following restrictions.
Suppose that $\mathbb{F}$ is an infinite field of characteristic not $2$ or $3$. Suppose that $A$ is a pseudo-composition algebra and $\eta=-1$ or a train algebra of rank~$3$ and $\eta\neq-1$. If $e$ is an idempotent in $A$, then
$A=A_1(e)\oplus A_{\eta}(e)\oplus A_{1/2}(e)$, where $A_1(e)$ is spanned by $e$, and products of eigenvectors from these subspaces obey the fusion law $\mathcal{PC}(\eta)$ \cite[Propositions~1.3 and 1.4]{W99}. Therefore, if $A$ is generated by idempotents, then $A$ is a $\mathcal{PC}(\eta)$-axial algebra.

Train algebras of rank 3 corresponding to $\eta=0$ in
Theorem~\ref{th:1} were investigated in~\cite{Wal88-2},
in particular, it was proved that these algebras are Jordan.
This implies that the class of primitive $\mathcal{PC}(0)$-axial algebras
is exactly the intersection of axial algebra classes corresponding to two
fusion laws in Table~\ref{t:1}: $\mathcal{PC}(0)$ and $\mathcal{J}(\frac{1}{2})$.

Note that until now, most of the papers on axial algebras have been devoted to two cases: axial algebras of Jordan type (fusion laws $\mathcal{J}(\eta)$) and Monster type
~\cite{survey}. In the proof of Theorem~\ref{th:1}, we use methods developed earlier for axial algebras of Jordan type and show how they can be applied to other classes of algebras. This gives us hope for a further expansion of axial algebras within the world of non-associative algebras.

In fact, to show that every $\mathcal{PC}(\eta)$-axial algebra is
a pseudo-composition algebra or train algebra of rank~$3$, we investigate subalgebras generated by three axes.
As a separate independent result, we describe such 3-generated algebras. Consider a primitive axis $x$ of an axial algebra $A$ over a field $\mathbb{F}$. If $y\in A$, then denote by $\varphi_x(y)$ the element of
$\mathbb{F}$ such that the projection of $y$ on $A_1(x)$ equals $\varphi_x(y)x$.
\begin{theorem}\label{th:2}
Suppose that $\mathbb{F}$ is a field of characteristic not two and $\eta\in\mathbb{F}\setminus\{\frac{1}{2}, 1\}$.
If $A$ is a $\mathcal{PC}(\eta)$-axial algebra over $\mathbb{F}$ generated by primitive axes $a$, $b$, and $c$, then $A$ is
the span of $a$, $b$, $c$, $ab$, $bc$, $ac$, $a(bc)$, and $b(ac)$: in particular $\dim A\leq8$. Furthermore, if $\alpha = \varphi_a(b)$,
$\beta =\varphi_b(c)$, $\gamma = \varphi_c(a)$, and $\psi=\varphi_a(bc)$, then
\begin{enumerate}[$(i)$]
\item $A_{\eta}(a)=\langle \alpha a+b-2ab, \gamma a+c-2ac, \psi a+bc-2a(bc)\rangle$;
\item $A_{1/2}(a)=\langle \alpha(\eta-1)a-\eta b+ab, \gamma(\eta-1)a-\eta c+ac, \psi(\eta-1)a-\eta bc+a(bc)$, \\
 $(2\eta^2-1)(2\psi-\beta)a-\eta\gamma b+(\eta-2\eta^2)\alpha c-bc+2b(ac)\rangle$.
\end{enumerate}
\end{theorem}
Note that if $\eta=-1$ then $\operatorname{char}\mathbb{F}\neq3$ since $\eta\neq\frac{1}{2}$. If $\eta\neq-1$ then the characteristic can be equal to $3$ in contrast to the assumptions of Theorem~\ref{th:1}.
We show in Tables~\ref{t:prod} and~\ref{t:prod2} how to multiply elements in the algebra from this theorem depending on $\eta$.
Moreover, we prove that if a basis for an 8-dimensional algebra has the table of product as in Tables~\ref{t:prod} or~\ref{t:prod2}, then  it is indeed a $\mathcal{PC}(\eta)$-axial algebra generated by three primitive axes.

Based on Theorem~\ref{th:2} and similarly to \cite[Problem~1(ii)]{hss2}, we formulate the following problem.

\begin{question}
Suppose that $A$ is a $\mathcal{PC}(\eta)$-axial $\mathbb{F}$-algebra generated by a finite set of primitive axes.
Is $A$ finite dimensional over $\mathbb{F}$?
\end{question}

This paper is organized as follows.
In Section~2, we described $\mathcal{PC}(\eta)$-axial algebra
generated by two primitive axes. This allows us to prove that
every $\mathcal{PC}(\eta)$-axial algebra admits a Frobenius form and is spanned by a set of axes. In Section~3, we prove Theorem~\ref{th:1}. Section~4 is devoted to the case of $\mathcal{PC}(\eta)$-axial algebra generated by three primitive axes and a proof of Theorem~\ref{th:2}.

\section{$\mathcal{PC}(\eta)$-axial algebras generated by two primitive axes}
In this section, we investigate 2-generated subalgebras of $\mathcal{PC}(\eta)$-axial algebras.
This allows us to prove that every
primitive $\mathcal{PC}(\eta)$-axial algebra admits a Frobenius form and is spanned by a set of axes. This proof follows the ideas outlined in \cite[Section 4]{hss}. Finally, we show that the fusion law for $\eta=-1$ is necessary stricter than it is in the definition.

Before moving on to the main part of the section, we want to emphasize that a description of 2-generated subalgebras
in a more general context of $J(\alpha,\beta)$-algebras
is available in \cite{Whybrow2}. Provided that $A$ is the linear span of its axes, the proof of the existence of a Frobenius form in \cite{hss} for axial algebras of Jordan type
also works for $J(\alpha,\beta)$-algebras.

Throughout, we suppose $\mathbb{F}$ is a field of characteristic not two and $\eta\in\mathbb{F}\setminus\{1,\frac{1}{2}\}$. Fix a $\mathcal{PC}(\eta)$-axial algebra $A$. Recall that if $a$ is a primitive axis, then $A=A_1(a)\oplus A_\frac{1}{2}(a)\oplus A_{\eta}(a)$, where $A_\lambda(a)=\{x\in A~|~ax=\lambda x\}$ and $A_1(a)$ is spanned by $a$. Moreover, for $\lambda,\mu\in\{1,\frac{1}{2},\eta\}$,
we have $A_\lambda(a)A_\mu(a)\subseteq A_{\lambda\ast\mu}(a)$, where the fusion law $\lambda\ast\mu$ is described in Table~\ref{t:1}. Note that
the fusion law $\mathcal{PC}(\eta)$ is $\mathbb{Z}_2$-graded. Indeed, if $a$ is an axis and
$A_+:=A_1(a)\oplus A_\eta(a)$ and $A_-:=A_\frac{1}{2}(a)$, then  $A = A_+\oplus A_-$, where $A_+A_+,A_-A_-\subseteq A_+$ and $A_+A_-\subseteq A_-$. This decomposition allows us to define
 an automorphism of $A$, called {\it the Miyamoto involution}, $\tau_a:A\rightarrow A$ such that if $x=x_1+x_2$, where $x_1\in A_+$ and $x_2\in A_-$, then
$x^{\tau_a}=x_1-x_2$.

Clearly, if $b$ is an axis in $A$, then so is
$b^{\tau_a}$ and $A_1(b)=A_1(b)^{\tau_a}$,
$A_\frac{1}{2}(b)=A_\frac{1}{2}(b)^{\tau_a}$, and $A_\eta(b)=A_\eta(b)^{\tau_a}$.
In what follows, we will use these properties without explanation.
First, we describe $\mathcal{PC}(\eta)$-axial algebras generated by two primitive axes. We will use double angular brackets $\langle\!\langle~
\rangle\!\rangle$ to indicate subalgebra generation, leaving single brackets for the linear span.
\begin{prop}\label{p:2gen}
Let $A$ be a $\mathcal{PC}(\eta)$-axial algebra.
Suppose that $a$ and $b$ are primtivie axes in $A$.
Write $b = \alpha a + b_{\eta} + b_{\frac{1}{2}}$, where $b_\eta\in A_\eta(a)$, $b_{\frac{1}{2}}\in A_\frac{1}{2}(a)$, and $\alpha\in\mathbb{F}$. Then the following statements hold
\begin{enumerate}[$(i)$]
\item
$b_{\eta} = \frac{1}{1-2\eta}(\alpha a + b - 2ab)$ and
$b_{\frac{1}{2}} = \frac{2}{1-2\eta}(\alpha (\eta-1) a  -\eta b +ab);$

\item $a(ab)=\frac{1}{2} \bigl((1-\eta)\alpha a - \eta b + (1+2\eta) ab\bigr)$;
\item $(ab)^2=\frac{1}{4}\Bigl(\bigl((1-\eta)(1-2\eta)\alpha-\eta(1+2\eta)\bigr)(a+b)+\bigl(2\alpha(1-\eta)(1+2\eta)+6\eta+4\eta^2\bigr)ab\Bigr)$;
\item the projection of $a$ on $A_1(b)$ equals $\alpha b$.
\end{enumerate}
\end{prop}
\begin{proof}
If $a=b$, then all claims are clear: in this case $\alpha=0$, $b_\eta=b_{\frac{1}{2}}=0$. So we further assume that $a\neq b$.

From the definition of $b_\eta$ and $b_{\frac{1}{2}}$, we get that $ab = \alpha a + \eta b_{\eta} + \frac{1}{2}b_{\frac{1}{2}}$.
Linear combinations of the expression for $ab$ with $b = \alpha a + b_{\eta} + b_{\frac{1}{2}}$ imply that
$$b_{\eta} = \frac{1}{1-2\eta}(\alpha a + b - 2ab),$$
$$b_{\frac{1}{2}} = \frac{2}{1-2\eta}(\alpha (\eta-1) a  -\eta b +ab).$$
This proves $(i)$. Using these expressions, we find that
$$a(ab) = \alpha a + \eta^2 b_{\eta} + \frac{1}{4} b_{\frac{1}{2}}=
\frac{1}{2}\bigl((1-\eta)\alpha a - \eta b + (1+2\eta) ab\bigr).$$
This proves the claim $(ii)$.

Write $a = \beta b + a_{\eta} + a_{\frac{1}{2}}$, where
$a_\eta\in A_\eta(b)$, $a_{\frac{1}{2}}\in A_\frac{1}{2}(b)$, and $\beta\in\mathbb{F}$. To prove $(iii)$, we show that $\alpha=\beta$.
Due to the symmetry of $a$ and $b$, we have
$$b(ab)=
\frac{1}{2} (- \eta a + (1-\eta)\beta b  + (1+2\eta) ab).$$

Substituting the expressions for $b$ and $ab$ with respect to $a$ into the right side, we find that
\begin{multline*}
b(ab)=
\frac{1}{2}\bigl(-\eta a + (1-\eta)\beta (\alpha a + b_{\eta} + b_{\frac{1}{2}})
+ (1+2\eta) (\alpha a + \eta b_{\eta} + \frac{1}{2}b_{\frac{1}{2}})\bigr)\\=
\frac{1}{2}\Bigl(\bigl(-\eta +(1-\eta)\alpha\beta + (1+2\eta)\alpha\bigr)a+
\bigl((1-\eta)\beta + (1+2\eta)\eta\bigr)b_{\eta}
+\bigl((1-\eta)\beta +\frac{1}{2}(1+2\eta)\bigr)b_{\frac{1}{2}}\Bigr).
\end{multline*}

Calculating this expression in a different way by multiplying $ab$ and $b$ written with respect to $a$, we get that
\begin{multline*}
b(ab) = (\alpha a + \eta b_{\eta} + \frac{1}{2}b_{\frac{1}{2}})(\alpha a + b_{\eta} + b_{\frac{1}{2}})\\=\alpha^2 a + \alpha \eta b_{\eta} +\frac{\alpha}{2}b_{\frac{1}{2}} +\alpha \eta^2 b_{\eta}
+\eta b_{\eta}^2+\eta b_{\eta}b_{\frac{1}{2}}+
\frac{\alpha}{4}b_{\frac{1}{2}}+\frac{1}{2}b_{\frac{1}{2}}b_{\eta}+\frac{1}{2}b_{\frac{1}{2}}^2.
\end{multline*}

Subtracting one of the representations of $b(ab)$ from the other and equating the summands from $A_\frac{1}{2}(a)$ to zero, we find that

\begin{equation}\label{alp=bet}
\bigl(\frac{1-\eta}{2}\beta +\frac{1+2\eta}{4}-\frac{3\alpha}{4}\bigr)b_{\frac{1}{2}}-\frac{2\eta+1}{2}b_{\eta}b_{\frac{1}{2}}=0.
\end{equation}

Similarly, considering the odd part in the equality $b\cdot b - b=0$, we find that
$$(\alpha-1)b_{\frac{1}{2}}+2b_{\eta}b_{\frac{1}{2}}=0.$$
Therefore,
$$b_{\eta}b_{\frac{1}{2}} = \frac{1-\alpha}{2}b_{\frac{1}{2}}.$$
Substituting this expression for $b_{\eta}b_{\frac{1}{2}}$ to \ref{alp=bet}, we get that
$$0=\bigl(\frac{1-\eta}{2}\beta +\frac{1+2\eta}{4}-\frac{3\alpha}{4}-\frac{(2\eta+1)}{2}\frac{(1-\alpha)}{2}\bigr)b_{\frac{1}{2}}=\frac{1}{2}(\eta-1)(\alpha-\beta)b_{\frac{1}{2}}.$$

By assumption, we have $\eta\neq1$. Suppose that $\alpha\neq\beta$.
Then $b_{\frac{1}{2}}=0$.  By $(i)$, we find that
$\alpha (\eta-1)a-\eta b+ab=0$. Due to the symmetry of $a$ and $b$,
we infer that $\alpha (\eta-1)a-\eta b=\beta(\eta-1)b-\eta a$.
This implies that $(\alpha (\eta-1)+\eta)a=(\beta(\eta-1)+\eta)b$.
Since $a\neq b$ and $\eta\neq1$, we find that $\alpha=\beta$; a contradiction.

It remains to express $(ab)^2$ in terms of $a$, $b$, and $ab$.
Consider the action of the Miyamoto involution $\tau_a$:
$$b^{\tau_a}=\alpha a + b_{\eta} - b_{\frac{1}{2}}=b-2b_{\frac{1}{2}}=
\frac{1}{(1-2\eta)}(4\alpha (1-\eta) a + (1+2\eta) b -4ab).$$

Using the fact that $b^{\tau_a}$ is an idempotent, we get that
$$\bigl(\frac{1}{(1-2\eta)}(4\alpha (1-\eta) a + (1+2\eta) b -4ab)\bigr)^2-\frac{1}{(1-2\eta)}(4\alpha (1-\eta) a + (1+2\eta) b -4ab)=0.$$

Opening the square and substituting expressions for $(ab)a$ and $(ab)b$ with $\beta=\alpha$, we find that
\begin{multline*}
\frac{1}{(1-2\eta)^2}\Bigl(16\alpha^2 (1-\eta)^2 a + (1+2\eta)^2 b +16(ab)^2
+8\alpha (1-\eta) (1+2\eta) ab \\-32\alpha (1-\eta) \frac{1}{2} \bigl((1-\eta)\alpha a - \eta b + (1+2\eta) ab)\bigr)-8(1+2\eta)\frac{1}{2}\bigl(- \eta a + (1-\eta)\alpha b  + (1+2\eta) ab\bigr)\Bigr)
\\-\frac{1}{(1-2\eta)}(4\alpha (1-\eta) a + (1+2\eta) b -4ab)=0.
\end{multline*}
This allows us to find the desired expression:
$$(ab)^2=\frac{1}{4}\Bigl(\bigl((1-\eta)(1-2\eta)\alpha-\eta(1+2\eta)\bigr)(a+b)+\bigl(2\alpha(1-\eta)(1+2\eta)+6\eta+4\eta^2\bigr)ab\Bigr).$$
\end{proof}

\begin{cor}
Suppose that $a$ and $b$ are primitive axes in a $\mathcal{PC}(\eta)$-axial algebra.  Then the subalgebra $\langle \langle a,b \rangle \rangle$,
generated by two primitive axes $a$ and $b$, is spanned by $a$, $b$, and $ab$.
\end{cor}
\begin{proof}
Clearly, $a$, $b$, and $ab$ are elements of $\langle \langle a,b \rangle \rangle$. On the other hand, Proposition~\ref{p:2gen} implies that
the set $\{a,b,ab\}$ is closed under multiplication.
\end{proof}

\begin{rmk}
The corresponding expression for $(ab)^2$ in \cite[Table~1]{Whybrow2} contains a misprint.
\end{rmk}

Note that if $\eta=-1$, then the product rules in this proposition indeed define a $\mathcal{PC}(-1)$-axial algebra generated by two primitive axes $a$ and $b$. This will follow from Corollary~\ref{cor:minus-one} below.
If $\eta\neq-1$, then additional restrictions must be added (see Proposition~\ref{p:train} and Corollary~\ref{cor:fusion-restricion}).

\begin{prop}\label{p:2dim}
Suppose that $a$ and $b$ are distinct primitive axes in a $\mathcal{PC}(\eta)$-axial algebra.  If the subalgebra
generated by $a$ and $b$ is $2$-dimensional, then either $ab=\frac{a+b}{2}$ or $\eta=-1$ and $ab=-(a+b)$. Conversely, if $A$ is a $2$-dimensional commutative algebra with a basis $\{a,b\}$ such that $a^2=a$, $b^2=b$ and
either $ab=\frac{a+b}{2}$ or $\eta=-1$ and $ab=-(a+b)$, then $A$ is a
$\mathcal{PC}(\eta)$-axial algebra generated by primitive axes $a$ and $b$.
\end{prop}
\begin{proof}
By Proposition~\ref{p:2gen}, $b=\alpha a+b_{\frac{1}{2}}+b_\eta$
and $a=\alpha b+a_{\frac{1}{2}}+a_\eta$, where $\alpha\in\mathbb{F}$,
$b_\frac{1}{2}\in A_{\frac{1}{2}}(a)$, $b_\eta\in A_{\eta}(a)$,
$a_\frac{1}{2}\in A_{\frac{1}{2}}(b)$, and $a_\eta\in A_{\eta}(b)$.

Assume that $b_\frac{1}{2}\neq0$ and $b_\eta\neq0$. Then
$a$, $b_\frac{1}{2}$, and $b_\eta$ are eigenvectors of $ad_a$, so
they are linearly independent; a contradiction with $\dim A=2$.
Therefore, $b_\frac{1}{2}=0$ or $b_\eta=0$ and, similarly,
$a_\frac{1}{2}=0$ or $a_\eta=0$. Suppose $b_\frac{1}{2}=0=a_{\eta}$.
It follows from Proposition~\ref{p:2gen}$(i)$ that $\alpha(\eta-1)a-\eta b+ab=0$ and $\alpha b+a-2ab=0$. Then
$2\alpha(\eta-1)a-2\eta b=-\alpha b-a$ and hence
$(2\alpha(\eta-1)+1)a=(2\eta-\alpha)b$. Since $a$ and $b$ are linearly independent, we infer that $\alpha=2\eta$ and $0=2\alpha(\eta-1)+1$.
Therefore, $0=4\eta^2-4\eta+1$; a contradiction with $\eta\neq\frac{1}{2}$.

Due to the symmetry of $a$ and $b$, we derive that either $b_\eta=a_\eta=0$ or $b_\frac{1}{2}=a_\frac{1}{2}=0$.
In the first case, $\alpha a+b-2ab=\alpha b+a-2ab$, so $(\alpha-1)a=(\alpha-1)b$. This implies that $\alpha=1$ and hence $0=b_\eta=a+b-2ab=0$. In the second case, $\alpha(\eta-1)a-\eta b+ab=\alpha (\eta-1)b-\eta a+ab$, so $(\alpha(\eta-1)+\eta)a=(\alpha(\eta-1)+\eta)b$. This implies that $\alpha=\frac{-\eta}{\eta-1}$ and hence $0=b_\frac{1}{2}=\frac{-\eta}{\eta-1}(\eta-1)a-\eta b+2ab=-\eta a-\eta b+ab$. Thus, we find that either $ab=\frac{a+b}{2}$ or $ab=\eta(a+b)$.

Conversely, suppose that $A$ is a $2$-dimension commutative algebra with a basis $\{a,b\}$ such that $a^2=a$, $b^2=b$. First, we consider the case when $ab=\eta(a+b)$.
Then $a(\eta a+(\eta-1)b)=\eta a+(\eta-1)ab=(\eta+\eta^2-\eta)a+(\eta^2-\eta)b=\eta(\eta a+(\eta-1)b)$. So $\eta a+(\eta-1)b$ is an eigenvector of $ad_a$ associated with $\eta$. Therefore, $A=A_1(a)\oplus A_\eta(a)$ and, similarly, $A=A_1(b)\oplus A_{\eta}(b)$.
Since $A_{1/2}(a)=0$, we see that $a$ is a primitive $\mathcal{PC}(\eta)$-axis in $A$
if and only if  $A_{\eta}(a)A_{\eta}(a)\subseteq A_1(a)$.
We know that $A_{\eta}(a)=\langle \eta a+(\eta-1) b\rangle$. Now
$(\eta a+(\eta-1) b)^2=\eta^2 a+(\eta-1)^2b+2(\eta^2-\eta)ab=(2\eta^3-\eta^2)a+(2\eta^3-\eta^2-2\eta+1)b$.
So $A_{\eta}(a)A_{\eta}(a)\subseteq A_1(a)$  is equivalent to the equality $0=2\eta^3-\eta^2-2\eta+1=(1-\eta^2)(1-2\eta)$.
Since $\eta\neq1,\frac{1}{2}$, in this case $a$ and $b$ are primitive $\mathcal{PC}(\eta)$-axis if and only if $\eta=-1$.

Let now $ab=\frac{a+b}{2}$. Then $a(a-b)=a^2-\frac{a+b}{2}=\frac{a-b}{2}$. So $a-b$ is an eigenvector of $ad_a$ associated with $\frac{1}{2}$. Therefore, $A=A_1(a)\oplus A_{1/2}(a)$ and, similarly, $A=A_1(b)\oplus A_{1/2}(b)$.
To check that $a$ is a $\mathcal{PC}(\eta)$-axis in $A$, it is enough to show that $A_{1/2}(a)A_{1/2}(a)\subseteq A_\eta(a)\oplus A_1(a)=A_1(a)$.
We know that $A_{1/2}(a)=\langle a-b\rangle$. Since $(a-b)^2=a^2-2ab+b^2=a-a-b+b=0$, we infer that $a$ is a primitive $\mathcal{PC}(\eta)$-axis. Similarly, we see that
$b$ is a primitive $\mathcal{PC}(\eta)$-axis and hence $A$ is a primitive $\mathcal{PC}(\eta)$-axial algebra.
\end{proof}

\begin{prop}\label{p:basis}
Let $A$ be a $\mathcal{PC}(\eta)$-axial algebra.
Then $A$ is the linear span of its set of primitive axes. In particular, there exists a basis consisting of primitive axes.
\end{prop}
\begin{proof}
Suppose that $A$ is generated by two primitive axes $a$ and $b$. By Proposition \ref{p:2gen},
$A$ is spanned by $a$, $b$, and $ab$. In the proof of
Proposition \ref{p:2gen}, it was shown that
$$b^{\tau_a}=\frac{1}{(1-2\eta)}(4\alpha (1-\eta) a + (1+2\eta) b -4ab).$$
Therefore, $A$ is spanned by three axes $a$, $b$, and $b^{\tau_a}$.

Now move to the general case for $A$. Denote by $B$ the linear span of the set of primitive axes from $A$. We prove by induction on $n$ that every word $w(a_1,\ldots,a_n)$ being a product of $n$ primitive axes $a_1,\ldots,a_n\in A$ is in $B$.
Clearly, the statement is true when $n\leq 2$. Assume $n>2$.
Write $w=w_1 \cdot w_2$, where $w_1$ and $w_2$ have smaller lengths. By induction, we may write $w_1=\alpha_1 b_1 + \ldots+\alpha_k b_k$ and $w_2=\beta_1 c_1 + \ldots +\beta_m c_m$ for some primitive axes $\{b_i\}_{1\leq i\leq k}$ and $\{c_j\}_{1\leq j\leq m}$ with coefficients $\alpha_i,\beta_j\in\mathbb{F}$.
Therefore, the product $w=w_1 \cdot w_2$ is a linear combination of products $b_ic_j$ involving two axes, which is proved to be in $B$. Thus, $A$ coincides with $B$.
\end{proof}

Proposition~\ref{p:basis} allows us to prove, as in \cite{hss}, the existence of a {\it Frobenius form} for every $\mathcal{PC}(\eta)$-axial algebra $A$. Recall that this means a non-zero symmetric bilinear form that associates with the algebra product, that is $(ab,c)=(a,bc)$ for all $a,b,c\in A$.

\begin{prop}\label{p:form}
Let $A$ be a $\mathcal{PC}(\eta)$-axial algebra. Then $A$ admits a unique Frobenius form such that $(a,a)=1$ for all primitive axes $a\in A$.
\end{prop}
\begin{proof}
If $a$ is a primitive axis in $A$ and $x\in A$, then
let $\varphi_a(x)$ be equal to the projection coefficient of $x$ onto $A_1(a)$.

Now we define the Frobenius form. By Proposition~\ref{p:basis},
there exists a basis $\mathcal{B}$ of $A$ consisting of primitive axes. For $a,b\in\mathcal{B}$, define $(a,b)=\varphi_a(b)$. Extending this by linearity, we obtain a bilinear form on $A$. In particular, if $u\in A$ and $a\in\mathcal{B}$, then $(a,u)=\varphi_a(u)$.

Since $(a,a)=1$ for every $a\in\mathcal{B}$, the form is non-zero.
It follows from Proposition~\ref{p:2gen} that $(a,b)=(b,a)$, so by linearity the form is symmetric.

Now we show that $(a,u)=\varphi_a(u)$ for every
primitive axis $a\in A$ and $u\in A$. Write
$u=\sum\limits_{b\in\mathcal{B}}\alpha_bb$.
Then, using Proposition~\ref{p:2gen}, we see that $$(a,u)=\sum\limits_{b\in\mathcal{B}}\alpha_b (a,b)=\sum\limits_{b\in\mathcal{B}}\alpha_b (b,a)=\sum\limits_{b\in\mathcal{B}}\alpha_b\varphi_b(a)=\sum\limits_{b\in\mathcal{B}}\alpha_b\varphi_a(b)=\varphi_a(u).$$

Note that the form is invariant under the action of the Miyamoto involution $\tau_a$ for every primitive axis  $a\in A$. Indeed, if
$b,c\in\mathcal{B}$ and $b=\alpha c+b_\eta+b_\frac{1}{2}$,
where $b_\eta\in A_\eta(c)$ and $b_\frac{1}{2}\in A_\frac{1}{2}(c)$, then
$b^{\tau_a}=\alpha c^{\tau_a}+b_\eta^{\tau_a}+b_\frac{1}{2}^{\tau_a}$. This implies that $(b,c)=\alpha=(b^{\tau_a},c^{\tau_a})$.

Now we verify the identity $(xy,z)=(x,yz)$, which is linear in $x$, $y$ and $z$.
We may assume that $y$ is a primitive axis, and that $x$ and $z$ are eigenvectors of $ad_y$ with eigenvalues $\mu$
and $\lambda$, respectively.

If $\mu = \lambda$, then $(xy,z)=(\mu x,z)=\mu(x,z)=\lambda(x,z)=(x,\lambda z)=(x,yz)$.
If $\mu \not = \lambda$, then it is sufficient to show that $(x,z)=0$. So we are left to prove that the
decomposition $A=A_1(y)\oplus A_\eta(y)\oplus A_{\frac{1}{2}}(y)$ is orthogonal with respect to the form.

If $x\in A_{\eta}(y) \oplus A_{\frac{1}{2}}(y)$, then $(y,x)=\varphi_y(x)=0$.
By the symmetry, we may assume that $x \in A_{\eta}(y)$ and $z\in A_{\frac{1}{2}}(y)$.
Now  $(x,z)=(x^{\tau_a},z^{\tau_a})=(x,-z)=-(x,z)$, therefore $(x,z)=0$. Thus, the form is Frobenius.

Conversely, suppose that a Frobenius form $(\cdot,\cdot)$ on $A$ satisfies the condition $(a,a)=1$ for all primitive axes $a\in A$.
Since the form is Frobenius, we see that the summands in the decomposition $A=A_1(a)\oplus A_\eta(a)\oplus A_{\frac{1}{2}}(a)$ are orthogonal with respect to the form.
Take $a\in\mathcal{B}$ and $u\in A$. Then $u=\varphi_a(u)a+u_\eta+u_\frac{1}{2}$, where $u_\eta\in A_\eta(a)$ and $u_\frac{1}{2}\in A_\frac{1}{2}(a)$. Now $(a,u)=(a,\varphi_a(u)a)+(a,u_\eta)+(a,u_\frac{1}{2})=\varphi_a(u)$. Therefore, this form coincides with the constructed above Frobenius form. The uniqueness is proved.
\end{proof}

\begin{prop}\label{p:train}
Suppose that $\eta\neq-1$ and $A$ is a $\mathcal{PC}(\eta)$-axial algebra. Denote the Frobenius form on $A$ from Proposition~\ref{p:form} by $(\cdot,\cdot)$. Then the following statements hold.
\begin{enumerate}[$(i)$]
\item If $a$ and $b$ are primitive axes in $A$, then $(a,b)=1$.
\item If $a$ is an axis in $A$, then the map $w_a:A\rightarrow\mathbb{F}$ defined by $w_a(x)=(a,x)$, is an $\mathbb{F}$-algebra homomorphism.
Moreover, for every primitive axis $b\in A$ we have $w_a(x)=w_b(x)$ and for all $x,y\in A$ we have $(x,y)=w_a(xy)$.
\end{enumerate}
\end{prop}
\begin{proof}
Suppose that $a$ and $b$ are primitive axes in $A$.
Prove that $(a,b)=1$. This is true if $a=b$, so we can assume that $a\neq b$.

Write $b=\alpha a+b_\eta+b_\frac{1}{2}$, where $\alpha\in\mathbb{F}$, $b_\eta\in A_\eta(a)$, and $b_\frac{1}{2}\in A_\frac{1}{2}(a)$. In the proof of Proposition~\ref{p:form} we see that $\alpha=(a,b)=w_a(b)$.
It is shown in Proposition~\ref{p:2gen} that $b_{\eta} = \frac{1}{1-2\eta}(\alpha a + b - 2ab)$.
Using Proposition~\ref{p:2gen}, we find $((1-2\eta)b_{\eta})^2:$
\begin{multline*}
(\alpha a + b - 2ab)^2=\alpha^2a+b+4(ab)^2+2\alpha ab-4\alpha a(ab)-4b(ab)=\alpha^2 a+b+4(ab)^2+2\alpha ab\\-2\alpha\bigl((1-\eta)\alpha a - \eta b + (1+2\eta) ab\bigr)-2\bigl((1-\eta)\alpha b - \eta a + (1+2\eta) ab\bigr)\\=
(\alpha^2(2\eta-1)+2\eta)a+(1+2\alpha\eta-2\alpha(1-\eta) )b+(-4\alpha\eta-2-4\eta)ab+4(ab)^2.
\end{multline*}
Substituting the value of $(ab)^2$ from Proposition~\ref{p:2gen},
we find that
\begin{multline*}
(\alpha a + b - 2ab)^2=
\bigl(\alpha^2(2\eta-1)+2\eta+(1-\eta)(1-2\eta)\alpha-\eta(1+2\eta)\bigr)a\\+
\bigl(1+2\alpha(2\eta-1)+(1-\eta)(1-2\eta)\alpha-\eta(1+2\eta)\bigr)b\\+
\bigl(2\alpha(1-\eta)(1+2\eta)+6\eta+4\eta^2-4\alpha\eta-2-4\eta\bigr)ab
\\=
(2\eta-1)(\alpha^2+\alpha(\eta-1)-\eta)a+
(2\eta-1)(2\alpha+\alpha(\eta-1)-1-\eta)b\\+
\bigl(2\alpha(-2\eta^2+\eta+1-2\eta)+(2\eta-1)(2\eta+2)\bigr)ab=
(2\eta-1)(\alpha-1)\bigl( (\alpha+\eta)a+(1+\eta)b-2(\eta+1)ab\bigr).
\end{multline*}
By the fusion rules, we know that $((1-2\eta)b_{\eta})^2\in A_1(a)$, so $$(2\eta-1)(\alpha-1)\bigl( (\alpha+\eta)a+(1+\eta)b-2(\eta+1)ab\bigr)\in A_1(a).$$
Suppose that $\alpha\neq1$.
Then, since $1+\eta\neq0$, we infer that $b-2ab\in A_1(a)$.
On the other hand, $b-2ab=-\alpha a+(1-2\eta)b_\eta$.
Therefore, we have $b_\eta=0$. So $2ab=b+\alpha a$. By the symmetry of $a$ and $b$, we have $b+\alpha a=a+\alpha b$, which is equivalent to $(1-\alpha)(a-b)=0$. Since $a\neq b$, we infer that $\alpha=1$; a contradiction. Thus, if $a$ and $b$ are primitive axes, then $(a,b)=1$.

We know that there exists a basis $\mathcal{B}$ of $A$ consisting of primitive axes. Suppose that $x,y\in A$ and $x=\sum\lambda_i a_i$,
$y=\sum\mu_i a_i$, where $\lambda_i,\mu_i\in\mathbb{F}$,
$a_i\in\mathcal{B}$, and $i$ runs over a finite set of indices.

First show that $w_a(x)=w_b(x)$:
$$w_a(x)-w_b(x)=(a,x)-(b,x)=(a,\sum\limits_i\lambda_i a_i)-
(b,\sum\limits_i\lambda_i a_i)=\sum\limits_i\lambda_i-\sum\limits_i\lambda_i=0.$$

Clearly, $w_a(x)$ is an $\mathbb{F}$-linear map.
It remains to prove that $w_a(xy)=w_a(x)w_a(y)$ and $(x,y)=w_a(xy)$.
Using that $w_a(a_ia_j)=w_{a_i}(a_ia_j)$ for all $i$ and $j$, we find that
\begin{multline*}
w_a(xy)=(a,xy)=(a,\sum\limits_{i,j}\lambda_i\mu_j a_i a_j)=
\sum\limits_{i,j}\lambda_i\mu_j(a, a_i a_j)=\sum\limits_{i,j}\lambda_i\mu_j(a_i, a_i a_j)\\ \sum\limits_{i,j}\lambda_i\mu_j(a_ia_i, a_j)=\sum\limits_{i,j}\lambda_i\mu_j(a_i,a_j)=\sum\limits_{i,j}\lambda_i\mu_j.
\end{multline*}
On the other hand,
$$w_a(x)w_a(y)=(a,x)(a,y)=(a,\sum\limits_i\lambda_ia_i)(a,\sum\limits_j\mu_ja_j)=(\sum\limits_i\lambda_i)(\sum\limits_j\mu_j)=\sum\limits_{i,j}\lambda_i\mu_j.
$$
Therefore, $w_a(xy)=w_a(x)w_a(y)$. Similarly, we see that $(x,y)=\sum\limits_{i,j}\lambda_i\mu_j=w_a(xy)$.
\end{proof}

\begin{cor}
If $\eta\neq-1$, then every $\mathcal{PC}(\eta)$-axial algebra
is a baric algebra.
\end{cor}
It also follows from Proposition~\ref{p:train} that in the case $\eta\neq-1$, the fusion rules are stricter than they are supposed by definition.
\begin{cor}\label{cor:fusion-restricion}
If $\eta\neq-1$ and $a$ is a primitive axis in a $\mathcal{PC}(\eta)$-axial algebra $A$, then $A_\eta(a)^2=\{0\}$ and $A_\frac{1}{2}(a)^2\subseteq A_\eta(a)$.
\end{cor}
\begin{proof}
Define as in Proposition~\ref{p:train} the map $w:A\rightarrow\mathbb{F}$ by $w(x)=(a,x)$, where $x\in A$.
If $x,y\in A_\eta(a)$, then $w(x)=w(y)=0$ since $A_\eta(a)$ and $A_1(a)$ are orthogonal with respect to the form.
Therefore, we find that $w(xy)=w(x)w(y)=0$.
On the other hand, by the fusion rules, we have $xy\in A_1(a)$ and hence $xy=(xy,a)a=w(xy)a$. So $xy=0$.

Similarly, if $x,y\in A_\frac{1}{2}(a)$, then $w(x)=w(y)=0$
and hence $w(xy)=w(x)w(y)=0$. On the other hand, by the fusion rules, we have $xy\in A_1(a)\oplus A_\eta(a)$ and hence $xy=(xy,a)a+u$, where $u\in A_\eta(a)$. Since $w(xy)=0$, we infer that $xy=u\in A_\eta(a)$.
\end{proof}

In the case of axial algebras of Jordan type, subgroups of the automorphism group generated by Miyamoto involutions play an important role (see, e.g., \cite{hss2}).
As a consequence of Proposition 1, in the conclusion of this section we prove the following assertion, which is of independent interest and can be used in further research.

\begin{lem}
Suppose that $a$ and $b$ are primitive axes in a $\mathcal{PC}(\eta)$-axial algebra.
If $\dim\langle\langle a,b\rangle\rangle=3$, then $\tau_a$ and $\tau_b$ have the following matrices on $\langle\langle a,b\rangle\rangle$ with respect to the basis $a$, $b$, and $ab$\footnote{we suppose that linear transformations act on the right}:
$$[\tau_a]= \frac{1}{1-2\eta} \begin{pmatrix}
1-2\eta & 0 & 0 \\
4\alpha (1-\eta) & 1+2\eta & -4 \\
2\alpha (1-\eta) & 2\eta & -1-2\eta
\end{pmatrix},
[\tau_b]= \frac{1}{1-2\eta} \begin{pmatrix}
1+2\eta & 4\alpha (1-\eta) & -4 \\
0 & 1-2\eta & 0 \\
2\eta & 2\alpha (1-\eta) & -1-2\eta \\
\end{pmatrix},$$
\end{lem}
\begin{proof}
Clearly, $a^{\tau_a}=a$.
We see in the proof of Proposition~\ref{p:2gen}, that
$$b^{\tau_a}=\frac{1}{(1-2\eta)}(4\alpha (1-\eta) a + (1+2\eta) b -4ab).$$
Now we find that
\begin{multline*}
(ab)^{\tau_a}=ab^{\tau_a}=\frac{1}{(1-2\eta)}(4\alpha (1-\eta) a + (1+2\eta) b -4ab)a
\\=\frac{1}{(1-2\eta)}(4\alpha (1-\eta) a + (1+2\eta) ab -4a(ab)).
\end{multline*}
Substituting the expression for $a(ab)$ from Proposition~\ref{p:2gen},
we get that
$$(ab)^{\tau_a}=\frac{1}{1-2\eta}(2\alpha (1-\eta) a +2\eta b -(1+2\eta) ab).$$
Therefore, we find all elements of the matrix for $\tau_a$. By the symmetry of $a$ and $b$,
we get the matrix for $\tau_b$.
\end{proof}

\section{Proof of Theorem~\ref{th:1}}
Throughout this section we suppose that
$\mathbb{F}$ is a field of characteristic not two and
$A$ is a primitive $\mathcal{PC}(\eta)$-axial algebra over $\mathbb{F}$, where $\eta\in\mathbb{F}\setminus\{1,\frac{1}{2}\}$. If $\eta=-1$, then we also assume that $\operatorname{char}(\mathbb{F})\neq3$ since otherwise $\eta=-1=\frac{1}{2}$. We denote by $(\cdot,\cdot)$ the Frobenius form on $A$ from Proposition~\ref{p:form}. To prove that $A$ is a pseudo-composition algebra if $\eta=-1$, we show that a linearization of the identity $x^3=(x,x)x$ holds on $A$.
If $\operatorname{char}(\mathbb{F})\neq3$, then the linearization implies that main identity. The same strategy is used for the case $\eta\neq-1$. The proof is split into several assertions.

\begin{lem}\label{l:a(ax)}
Suppose that $a$ is an axis in $A$ and $x\in A$.
Then $a(ax)=(a,x)\frac{1-\eta}{2}a-\frac{\eta}{2}x+\frac{1+2\eta}{2}ax$.
\end{lem}

\begin{proof}
The proof is exactly the same as for the product $a(ab)$ in Proposition~\ref{p:2gen}.
\end{proof}

\begin{lem}\label{l:b_eta-c_eta}
Suppose that $a$ is a primitive axis in $A$ and $x,y\in A$.
Denote $\alpha=(a,x), \gamma=(a,y)$, $\beta=(x,y)$, and $\psi=(a,xy)$. Write $x=\alpha a+x_\eta+x_\frac{1}{2}$ and
$y=\gamma a+y_\eta+y_\frac{1}{2}$, where $x_\eta, y_\eta\in A_\eta(a)$ and $x_\frac{1}{2},y_\frac{1}{2}\in A_\frac{1}{2}(a)$.
Then
\begin{enumerate}[$(i)$]
\item $(x_\eta,y_\eta)=\frac{1}{1-2\eta}(\alpha\gamma+\beta-2\psi)$;
\item $(x_\frac{1}{2}, y_\frac{1}{2})=\frac{2}{1-2\eta}((\eta-1)\alpha\gamma-\eta\beta+\psi)$.
\end{enumerate}
\end{lem}
\begin{proof}
Clearly, $ax=\alpha a+\eta x_\eta+\frac{1}{2}x_\frac{1}{2}$ and hence $x_\eta=\frac{1}{1-2\eta}(\alpha a+x-2ax)$. Similarly, we find that $y_\eta=\frac{1}{1-2\eta}(\gamma a+y-2ay)$.
Therefore,
\begin{multline*}
(1-2\eta)^2(x_\eta,y_\eta)=(\alpha a+x-2ax, \gamma a+y-2ay)\\=
\alpha\gamma(a,a)+\alpha(a,y)-2\alpha(a,ay)+\gamma(x,a)+(x,y)-2(x,ay)-2\gamma(ax,a)-2(ax,y)+4(ax,ay)
\\=\alpha\gamma+\alpha\gamma-2\alpha\gamma+\alpha\gamma+\beta-2\psi-2\alpha\gamma-2\psi+4(ax,ay)=-\alpha\gamma+\beta-4\psi+4(ax,ay).
\end{multline*}
Using Lemma~\ref{l:a(ax)}, we see that
\begin{multline*}
2(ax,ay)=(2a(ax),y)=(\alpha(1-\eta)a-\eta x+(1+2\eta)ax,y)\\=\alpha(1-\eta)(a,y)-\eta(x,y)+(1+2\eta)(ax,y)=\alpha\gamma(1-\eta)-\eta\beta+(1+2\eta)\psi.
\end{multline*}
This implies that
$$(1-2\eta)^2(x_\eta,y_\eta)=(1-2\eta)\alpha\gamma+(1-2\eta)\beta+(4\eta-2)\psi,$$
and hence $(x_\eta,y_\eta)=\frac{1}{(1-2\eta)}(\alpha\gamma+\beta-2\psi)$, as required.

Since $\beta=(x,y)=(\alpha a+x_\eta+x_\frac{1}{2},\gamma a+y_\eta+y_\frac{1}{2})=\alpha\gamma+(x_\eta,y_\eta)+(x_\frac{1}{2},y_\frac{1}{2})$,
we find that
$$(x_\frac{1}{2},y_\frac{1}{2})=\beta-\alpha\gamma-\frac{1}{1-2\eta}(\alpha\gamma+\beta-2\psi)=\frac{1}{1-2\eta}(2(\eta-1)\alpha\gamma-2\eta\beta+2\psi),$$
as claimed.
\end{proof}

\begin{lem}\label{l:3prod} Suppose that $a$ is a primitive axis and  $x,y\in A$.
Denote $\alpha=(a,x)$, $\beta=(x,y)$, $\gamma=(a,y)$, and $\psi=(a,xy)$.
Then
$$a(xy)+x(ay)+y(ax)=\bigl((1+\eta)(\psi-\alpha\gamma)-\eta\beta\bigr)a-\gamma\eta x-\alpha\eta y+(1+\eta)(xy+\gamma ax+\alpha ay).$$
\end{lem}
\begin{proof}
Write $x=\alpha a+x_\eta+x_\frac{1}{2}$ and
$y=\gamma a+y_\eta+y_\frac{1}{2}$, where $x_\eta,x_\eta\in A_\eta(a)$ and $y_\frac{1}{2}, y_\frac{1}{2}\in A_\frac{1}{2}(a)$.
Then
\begin{multline*}
a(xy)=a\bigl((\alpha a+x_\eta+x_\frac{1}{2})(\gamma a+y_\eta+y_\frac{1}{2})\bigr)\\=a\bigl(\alpha\gamma a+\alpha\eta y_\eta+\frac{\alpha}{2} y_\frac{1}{2}+\gamma\eta x_\eta+x_\eta y_\eta+x_\eta y_\frac{1}{2}+\frac{\gamma}{2} x_\frac{1}{2}+x_\frac{1}{2}y_\eta+x_\frac{1}{2}y_\frac{1}{2}\bigr)\\=\alpha\gamma a+\alpha\eta^2 y_\eta+\frac{\alpha}{4} y_\frac{1}{2}+\gamma\eta^2 x_\eta+x_\eta y_\eta+\frac{1}{2}x_\eta y_\frac{1}{2}+\frac{\gamma}{4} x_\frac{1}{2}+\frac{1}{2}x_\frac{1}{2}y_\eta+a(x_\frac{1}{2}y_\frac{1}{2}).
\end{multline*}
Similarly, we see that
\begin{multline*}
x(ay)=(\alpha a+x_\eta+x_\frac{1}{2})(\gamma a+\eta y_\eta+\frac{1}{2}y_\frac{1}{2})\\=
\alpha\gamma a+\alpha\eta^2 y_\eta+\frac{\alpha}{4} y_\frac{1}{2}+\gamma\eta x_\eta+\eta x_\eta y_\eta+\frac{1}{2}x_\eta y_\frac{1}{2}+\frac{\gamma}{2}x_\frac{1}{2}+\eta x_\frac{1}{2}y_\eta+\frac{1}{2}x_\frac{1}{2}y_\frac{1}{2}.
\end{multline*}
By the symmetry of $x$ and $y$,
$$
y(ax)=\alpha\gamma a+\gamma\eta^2 x_\eta+\frac{\gamma}{4} x_\frac{1}{2}+\alpha\eta y_\eta+\eta x_\eta y_\eta+\frac{1}{2}y_\eta x_\frac{1}{2}+\frac{\alpha}{2} y_\frac{1}{2}+\eta y_\frac{1}{2}x_\eta+\frac{1}{2}x_\frac{1}{2}y_\frac{1}{2}.
$$
These equalities imply that
\begin{multline*}
a(xy)+x(ay)+y(ax)-(1+\eta)xy\\=\alpha\gamma(2-\eta) a+\alpha\eta^2 y_\eta+\gamma\eta^2 x_\eta+\alpha\frac{1-\eta}{2}y_\frac{1}{2}+\gamma\frac{1-\eta}{2} x_\frac{1}{2}+\eta x_{\eta}y_{\eta}+a(x_\frac{1}{2}y_\frac{1}{2})-\eta x_\frac{1}{2}y_\frac{1}{2}.
\end{multline*}
Since $x_\eta y_\eta\in A_1(a)$, we have
$x_\eta y_\eta=(a,x_\eta y_\eta) a=(ax_\eta,y_\eta) a=\eta(x_\eta,y_\eta) a$. Applying Lemma~\ref{l:b_eta-c_eta},
we find that $\eta x_\eta y_\eta=\frac{\eta^2}{1-2\eta}(\alpha\gamma+\beta-2\psi) a$.

Since $x_\frac{1}{2}y_\frac{1}{2}\in A_1(a)\oplus A_\eta(a)$, we have
$$a(x_\frac{1}{2}y_\frac{1}{2})-\eta x_\frac{1}{2}y_\frac{1}{2}=(1-\eta)(a,x_\frac{1}{2}y_\frac{1}{2}) a=
(1-\eta)(ax_\frac{1}{2},y_\frac{1}{2})a=\frac{1-\eta}{2}(x_\frac{1}{2},y_\frac{1}{2})a.$$
 Applying Lemma~\ref{l:b_eta-c_eta},
we see that $a(x_\frac{1}{2}y_\frac{1}{2})-\eta x_\frac{1}{2}y_\frac{1}{2}=\frac{1-\eta}{1-2\eta}((\eta-1)\alpha\gamma-\eta\beta+\psi) a$.

Note that
\begin{multline*}
(1+\eta) ax-\eta x=(1+\eta)\alpha a+(1+\eta)\eta x_\eta+\frac{1+\eta}{2}x_\frac{1}{2}-\alpha\eta a-\eta x_\eta-\eta x_\frac{1}{2}\\=\alpha a+\eta^2 x_{\eta}+\frac{1-\eta}{2} x_\frac{1}{2}.
\end{multline*}

Similarly, $(1+\eta) ay-\eta y=\gamma a+\eta^2 y_{\eta}+\frac{1-\eta}{2} y_\frac{1}{2}$. Using these equalities, we find that
\begin{multline*}
a(xy)+x(ay)+y(ax)-(1+\eta) xy-\gamma((1+\eta) ax-\eta x)-\alpha((1+\eta) ay-\eta y)\\=\alpha\gamma(2-\eta) a-2\alpha\gamma a+\eta x_{\eta}y_{\eta}+a(x_\frac{1}{2}y_\frac{1}{2})-\eta x_\frac{1}{2}y_\frac{1}{2}\\=
\bigl(-\alpha\gamma\eta+\frac{\eta^2}{1-2\eta}(\alpha\gamma+\beta-2\psi)+ \frac{1-\eta}{1-2\eta}((\eta-1)\alpha\gamma-\eta\beta+\psi)\bigr) a\\=\frac{1}{1-2\eta}\bigl(-\alpha\gamma\eta(1-2\eta)+\eta^2(\alpha\gamma+\beta-2\psi)+(1-\eta)((\eta-1)\alpha\gamma-\eta\beta+\psi)\bigr)a\\=
\frac{1}{1-2\eta}\bigl( (2\eta^2+\eta-1)\alpha\gamma+(2\eta^2-\eta)\beta+(1-\eta-2\eta^2)\psi\bigr)a=\bigl((1+\eta)(\psi-\alpha\gamma)-\eta\beta\bigr)a.
\end{multline*}
Therefore,
$$a(xy)+x(ay)+y(ax)=\bigl((1+\eta)(\psi-\alpha\gamma)-\eta\beta\bigr) a-\gamma\eta x-\alpha\eta y+(1+\eta)(xy+\gamma ax+\alpha ay).
$$
\end{proof}

\begin{cor}\label{c:3prod} Suppose that $a$ is a primitive axis in $A$ and $x,y\in A$. The following statements hold.
\begin{enumerate}[$(i)$]
\item If $\eta=-1$, then $a(xy)+x(ay)+y(ax)=(x,y) a+(a,y) x+(a,x) y.$
\item If $\eta\neq-1$, then $$a(xy)+x(ay)+y(ax)=
(1+\eta)\bigl(w(x)ay+w(y)ax+xy\bigr)-\eta\bigl(w(xy)a+w(y)x+w(x)y\bigr),$$
where $w(x)=(a,x)$, $w(y)=(a,y)$, and $w(xy)=(a,xy)$.
\end{enumerate}
\end{cor}
\begin{proof}
If $\eta=-1$, the assertion follows from Lemma~\ref{l:3prod}. If $\eta\neq-1$, then Proposition~\ref{p:train} implies that $w(xy)=w(x)w(y)=(x,y)$. Applying these equalities and Lemma~\ref{l:3prod}, we get the identity for this case.
\end{proof}

\begin{lem}\label{l:eta=-1}
Suppose that $\eta=-1$. If $x,y,z\in A$, then
$x(yz)+y(zx)+z(xy)=(y,z)x+(x,z)y+(y,z)z$. In particular, $x^3=(x,x)x$ and hence $A$ is a pseudo-composition algebra.
\end{lem}
\begin{proof}
By Proposition~\ref{p:basis}, $A$ is the span of some set of axes $\{a_i\}_{i\in I}$. Write $x=\sum\limits_{i=1}^n\lambda_i\cdot a_i$,
where $\lambda_i\in\mathbb{F}$ for $1\leq i\leq n$.
By Corollary~\ref{c:3prod}, we have
$$a_i(yz)+y(a_iz)+z(a_iy)=(y,z)a_i+(a_i,z)y+(a_i,y)z.$$
Multiplying both sides by $\lambda_i$, we get
$$(\lambda_ia_i)(yz)+y((\lambda_ia_i)z)+z((\lambda_ia_iy)=(y,z)(\lambda_ia_i)+(\lambda_ia_i,z)y+(\lambda_ia_i,y)z.$$
Summing up the equalities for all $i$, we obtain the required equality $x(yz)+y(zx)+z(xy)=(y,z)x+(x,z)y+(y,z)z$.
Finally, if $x=y=z$, then we have $3x^3=3(x,x)x$ and hence
$x^3=(x,x)x$ since $\operatorname{char}\mathbb{F}\neq3$.
\end{proof}

\begin{lem}
Suppose that $\eta\neq-1$.
Define the $\mathbb{F}$-algebra homomorphism from Proposition~\ref{p:train} by $w$.
Then for all $x,y,z\in A$ it is true that $$x(yz)+y(xz)+z(yx)=(\eta+1)\bigl(w(z)xy+w(x)yz+w(y)xz\bigr)-
\eta\bigl(w(yz)x+w(zx)y+w(xy)z\bigr).$$ In particular, if $\operatorname{char}\mathbb{F}\neq3$, then
$x^3=(\eta+1)w(x)x^2-\eta w(x)^2x$ and hence $A$ is a train algebra of rank $3$.
\end{lem}
\begin{proof}
By Proposition~\ref{p:basis}, $A$ is the span of some set of axes $\{a_i\}_{i\in I}$. Write $x=\sum\limits_{i=1}^n\lambda_i\cdot a_i$,
where $\lambda_i\in\mathbb{F}$ for $1\leq i\leq n$.
By Corollary~\ref{c:3prod}, we have
\begin{multline*}
a_i(yz)+y(a_iz)+z(ya_i)\\=
(1+\eta)\bigl(w(a_i)yz+w(y)a_iz+w(z)a_iy\bigr)-\eta\bigl(w(yz)a_i+w(a_iz)y+w(a_iy)z\bigr).
\end{multline*}
Multiplying both sides by $\lambda_i$, we find that
\begin{multline*}
(\lambda_ia_i)(yz)+y((\lambda_ia_i)z)+z(y(\lambda_ia_i))=
(1+\eta)\bigl(w(\lambda_ia_i)yz+w(y)(\lambda_ia_i)z+w(z)(\lambda_ia_i)y\bigr)\\-\eta\bigl(w(yz)(\lambda_ia_i)+w((\lambda_ia_i)z)y+w((\lambda_ia_i)y)z\bigr).
\end{multline*}

Summing up the equalities for all $i$, we get the required equality
$$x(yz)+y(xz)+z(yx)=(\eta+1)\bigl(w(x)yz+w(y)xz+w(z)xy\bigr)-\eta\bigl(w(yz)x+w(xz)y+w(xy)z\bigr).$$
If $x=y=z$ and $\operatorname{char}\mathbb{F}\neq3$, then we have
$3x^3=3(\eta+1)w(x)x^2-3\eta w(x)^2x$ and hence
$x^3=(\eta+1)w(x)x^2-\eta w(x)^2x$, as claimed.
\end{proof}

\section{$\mathcal{PC}(\eta)$-axial algebras generated by three primitive axes}
In this section, we prove Theorem~\ref{th:2}.
Throughout, we suppose that $\mathbb{F}$ is a field of characteristic not two and $A$ is a primitive $\mathcal{PC}(\eta)$-axial algebra over $\mathbb{F}$, where $\eta\in\mathbb{F}\setminus\{1,\frac{1}{2}\}$. We denote by $(\cdot,\cdot)$ the Frobenius form on $A$ from Proposition~\ref{p:form}. It follows from Proposition~\ref{p:train} that if $\eta\neq-1$, then there exists an $\mathbb{F}$-algebra homomorphism $w:A\rightarrow\mathbb{F}$ such that $w(xy)=(x,y)$ for all $x,y\in A$.

\begin{lem}\label{l:axay1}
Suppose that $a$ is a primitive axis in $A$ and $x,y\in A$.
If $\eta=-1$, then
\begin{multline*}
(ax)(ay)=\frac{1}{4}\bigl((6(a,xy)-3(x,y))a+(a,y)x+(a,x)y\\-2(a,y)ax-2(a,x)ay-xy+2x(ay)+2y(ax)\bigr).
\end{multline*}
\end{lem}
\begin{proof}
Write $x=\alpha a+x_{-1}+x_\frac{1}{2}$ and
$y=\beta a+y_{-1}+y_\frac{1}{2}$, where $\alpha=(a,x),\beta=(a,y)$, $x_{-1},y_{-1}\in A_{-1}(a)$, and $x_\frac{1}{2},y_\frac{1}{2}\in A_\frac{1}{2}(a)$.
By the fusion laws, we have
$x_{-1}y_{-1}=-(x_{-1}, y_{-1})a$. Since $x_{-1}=\frac{1}{3}(x-2ax+\alpha a)$ and $y_{-1}=\frac{1}{3}(y-2ay+\beta a)$,
we may express $(ax)(ay)$ using the product $x_{-1}y_{-1}$:
$$(ax)(ay)=\frac{1}{4}\bigl(-9(x_{-1}, y_{-1})a-(x+\alpha a)(y-2ay+\beta a)+2(ax)(y+\beta a)\bigr).$$
By Lemma~\ref{l:b_eta-c_eta}, we know that $9(x_{-1}, y_{-1})=3(x,y)-6(a,xy)+3\alpha\beta.$
Using Lemma~\ref{l:a(ax)}, we see that
\begin{multline*}
(x+\alpha a)(y-2ay+\beta a)=xy-2x(ay)+\beta ax+\alpha ay-2\alpha a(ay)+\alpha\beta a\\=\alpha\beta a+\beta ax+\alpha ay + xy-2x(ay)-2\alpha(\beta a+\frac{1}{2}y-\frac{1}{2}ay)\\=-\alpha\beta a+\beta ax-\alpha y+2\alpha ay+xy-2x(ay).
\end{multline*}
Finally, we find that
$$2(ax)(y+\beta a)=2y(ax)+2\beta(\alpha a+\frac{1}{2}x-\frac{1}{2}ax)=2\alpha\beta a+\beta x-\beta ax+2y(ax).$$
Combining calculations, we infer that
$$(ax)(ay)=\frac{1}{4}\bigl((6(a,xy)-3(x,y))a+\beta x+\alpha y-xy-2\beta ax-2\alpha ay+2x(ay)+2y(ax)\bigr),$$
as claimed.
\end{proof}

\begin{lem}\label{l:axay2}
Suppose that $a$ is a primitive axis in $A$ and $x,y\in A$.
If $\eta\neq-1$, then
\begin{multline*}
(ax)(ay)=\frac{1}{4}\bigl((1-2\eta)w(xy)a-\eta w(y)x-\eta w(x)y\\+2\eta w(y)ax+2\eta w(x)ay-xy+2x(ay)+2y(ax)\bigr).
\end{multline*}

\end{lem}
\begin{proof}
The proof is similar to that of Lemma~\ref{l:axay1}.
\end{proof}

\begin{table}
\begin{center}
\begingroup
\setlength{\tabcolsep}{20pt}
\renewcommand{\arraystretch}{1.2}
{\small
$\begin{tabu}[h!]{|c|}
\hline
a\cdot a = a, b\cdot b=b, c\cdot c=c, a\cdot b = ab, a\cdot c=ac, b\cdot c=bc, \\ \hline
a\cdot bc = a(bc), b\cdot ac = b(ac), c\cdot ab =\beta a+\gamma b+\alpha c-a(bc)-b(ac),  \\ \hline
\begin{tabu}{@{}c@{}}
a\cdot ab = \alpha a+\frac{1}{2}b-\frac{1}{2}{ab}, a\cdot ac = \gamma a+\frac{1}{2}c-\frac{1}{2}{ac}, b\cdot ab = \alpha b+\frac{1}{2}a-\frac{1}{2}{ab}, \\
b\cdot bc = \beta b+\frac{1}{2}c-\frac{1}{2}{bc}, c\cdot ac =\gamma c+\frac{1}{2}a-\frac{1}{2}{ac}, c\cdot bc = \beta c+\frac{1}{2}b-\frac{1}{2}{bc},
\end{tabu}\\ \hline
a\cdot(a(bc))=\psi a+\frac{1}{2}bc-\frac{1}{2}a(bc),
b\cdot(b(ac))=\psi b+\frac{1}{2}ac-\frac{1}{2}b(ac), \\ \hline
\begin{tabu}{@{}c@{}}
b\cdot a(bc)=\frac{1}{4}( \beta a+(\gamma-2\psi)b-3\alpha c-2\beta ab+ 6\alpha bc-ac+2a(bc)+2b(ac)), \\
c\cdot a(bc)=\frac{1}{4}(3\beta a-\gamma b+(3\alpha-2\psi)c-ab+6\gamma bc-2\beta ac-2b(ac)), \\
a\cdot b(ac)=\frac{1}{4}( (\beta-2\psi)a+\gamma b-3\alpha c-2\gamma ab -bc+6\alpha ac+2a(bc)+2b(ac) ), \\
c\cdot b(ac)=\frac{1}{4}(-\beta a+3\gamma b+(3\alpha-2\psi)c-ab-2\gamma bc+6\beta ac-2a(bc)),
\end{tabu} \\ \hline
\begin{tabu}{@{}c@{}}
(ab)^2=\frac{1}{4}((6\alpha-1)(a+b)-(4\alpha+2)ab),
(bc)^2=\frac{1}{4}((6\beta-1)(b+c)-(4\beta+2)bc), \\
(ac)^2=\frac{1}{4}((6\gamma-1)(a+c)-(4\gamma+2)ac),
\end{tabu}\\ \hline
\begin{tabu}{@{}c@{}}
ab\cdot bc=\frac{1}{4}(3\beta a+(6\psi-\gamma)b+3\alpha c-2\beta ab-2\alpha bc-ac-2b(ac)), \\
bc\cdot ac=\frac{1}{4}(\beta a+\gamma b+(6\psi-3\alpha)c-ab-2\gamma bc-2\beta ac + 2a(bc)+2b(ac)), \\
ab\cdot ac=\frac{1}{4}( (6\psi-\beta)a+3\gamma b+3\alpha c-2\gamma ab -bc-2\alpha ac-2a(bc) ),
\end{tabu}\\ \hline
\begin{tabu}{@{}c@{}}
ab\cdot a(bc)=\frac{1}{8}\bigl((12\alpha\beta-2\beta+6\gamma-6\psi)a+(6\psi-2\beta)b-c-4(\beta+\psi)ab\\+(6\alpha+1)bc-2ac+(2-4\alpha)a(bc)\bigr), \\
ab\cdot b(ac)=\frac{1}{8}\bigl((6\psi-2\gamma)a+(12\alpha\gamma+6\beta-2\gamma-6\psi)b-c-4(\gamma+\psi)ab-2bc\\+(6\alpha+1)ac+(2-4\alpha)b(ac)\bigr), \\
bc\cdot a(bc)=\frac{1}{8}\bigl( (4\beta^2-2\beta+2)a+(1-6\beta)ab+8\psi bc\\+(1-6\beta)ac+(4\beta+2)a(bc) \bigr), \\
bc\cdot b(ac)=\frac{1}{8}\bigl(-a+(12\beta\gamma+6\alpha-2\gamma-6\psi)b+ (6\psi-2\gamma)c-2ab-4(\gamma+\psi)bc\\+(6\beta+1)ac+(2-4\beta)b(ac)\bigr), \\
ac\cdot a(bc)=\frac{1}{8}\bigl( (12\beta\gamma+6\alpha-2\beta-6\psi)a-b+(6\psi-2\beta)c-2ab\\+(6\gamma+1)bc-4(\beta+\psi)ac+(2-4\gamma)a(bc)\bigr), \\
ac\cdot b(ac)=\frac{1}{8}\bigl( (4\gamma^2-2\gamma+2)b+(1-6\gamma)ab+ (1-6\gamma)bc\\+8\psi ac +(4\gamma+2)b(ac) \bigr),
\end{tabu}\\ \hline
\begin{tabu}{@{}c@{}}
(a(bc))^2=\frac{1}{16}\bigl(
(36\alpha\beta-4\beta^2+36\beta\gamma-24\beta\psi-6\alpha+2\beta-6\gamma-12\psi-2)a+(1-6\beta)(b+c)\\+2(1-6\beta)(ab+ac)+(4\beta+24\psi+2)bc+(8\beta-16\psi+4)a(bc) \bigr), \\
(b(ac))^2=\frac{1}{16}\bigl(
(1-6\gamma)(a+c)+(36\alpha\gamma+36\beta\gamma-4\gamma^2-24\gamma\psi-6\alpha-6\beta+2\gamma-12\psi-2)b \\
+(2-12\gamma)(ab+bc)+(4\gamma+24\psi+2)ac+(8\gamma-16\psi+4)b(ac)\bigr), \\
a(bc)\cdot b(ac)=\frac{1}{16}\bigl(
(12\beta\psi-6\alpha\beta-2\beta^2-14\beta\gamma+\beta+6\gamma+6\psi+1)a\\+(12\gamma\psi-6\alpha\gamma-14\beta\gamma-2\gamma^2+6\beta+\gamma+6\psi+1)b\\+(18\alpha^2+6\alpha\beta+6\alpha\gamma-36\alpha\psi+3\alpha-6\psi-2)c\\+(6\gamma-6\alpha+6\beta-32\beta\gamma+12\psi-1)ab+
 (24\alpha\gamma-6\alpha+2\beta+6\gamma-12\psi-1)bc\\+(24\alpha\beta-6\alpha+6\beta+2\gamma-12\psi-1)ac\\+(8\gamma-12\alpha-4\beta+16\psi-4)a(bc)
+(8\beta-12\alpha-4\gamma+16\psi-4)b(ac)\bigr).
\end{tabu} \\ \hline
\end{tabu}$}
\caption{Table of products for $\eta=1$}\label{t:prod}
\endgroup
\end{center}
\end{table}

\begin{table}
\begin{center}
\begingroup
\setlength{\tabcolsep}{20pt}
\renewcommand{\arraystretch}{1.5}

{
$\begin{tabu}[h!]{|c||c|c|c|c|c|c|c|c|}
\hline
(,) &  a & b & c & ab & bc & ac & a(bc) & b(ac)  \\ \hline\hline
a &  1 & \alpha & \gamma & \alpha & \psi & \gamma & \psi &
\alpha\gamma+\frac{\beta-\psi}{2}  \\ \hline
b &  & 1 & \beta & \alpha & \beta & \psi & \alpha\beta+\frac{\gamma-\psi}{2} & \psi   \\ \hline
c &  &  & 1 & \psi & \beta & \gamma & \beta\gamma+\frac{\alpha-\psi}{2} & \beta\gamma+\frac{\alpha-\psi}{2} \\ \hline
ab &  &  & & \frac{2\alpha^2-\alpha+1}{2} & \frac{2\alpha\beta+\gamma-\psi}{2} & \frac{2\alpha\gamma+\beta-\psi}{2} & \frac{4\alpha\psi+2\beta+\psi-2\alpha\beta-\gamma}{4} & \frac{4\alpha\psi+2\gamma+\psi-2\alpha\gamma-\beta}{4}  \\ \hline
bc &  & &  &  & \frac{2\beta^2-\beta+1}{2} & \frac{2\beta\gamma+\alpha-\psi}{2} & \frac{(6\beta-1)(\alpha+\gamma)-(4\beta+2)\psi}{4} & \frac{4\beta\psi+2\gamma+\psi-2\beta\gamma-\alpha}{4}   \\ \hline
ac &  &  & &  &  & \frac{2\gamma^2-\gamma+1}{2} & \frac{4\gamma\psi+2\beta+\psi-2\beta\gamma-\alpha}{4} & \frac{(6\gamma-1)(\alpha+\beta)-(4\gamma+2)\psi}{4}  \\ \hline
a(bc) &  &  & &  & &  & \begin{tabu}{@{}c@{}}\frac{4\beta^2-6\beta\gamma-6\alpha\beta+4\beta\psi}{8}\\\frac{+8\psi^2+\alpha-2\beta+\gamma+2\psi+2}{8} \end{tabu}  & \begin{tabu}{@{}c@{}} \frac{8\alpha\beta\gamma+6\alpha^2-6\alpha\psi-2\beta^2}{8} \\ \frac{+6\beta\gamma+2\beta\psi-2\gamma^2+2\gamma\psi}{8} \\ \frac{-4\psi^2-2\alpha+\beta+\gamma-\psi-1}{8}  \end{tabu} \\ \hline
b(ac) &  &  & &  & & & &  \begin{tabu}{@{}c@{}}
\frac{4\gamma^2-6\alpha\gamma-6\beta\gamma+4\gamma\psi}{8} \\
\frac{+8\psi^2+\alpha+\beta-2\gamma+2\psi+2}{8} \end{tabu}\\ \hline
\end{tabu}$}
\caption{the Gram matrix}\label{t:gram}
\endgroup
\end{center}
\end{table}

\begin{table}
\begin{center}
\begingroup
\setlength{\tabcolsep}{20pt}
\renewcommand{\arraystretch}{1.3}
{\small
$\begin{tabu}[h!]{|c|}
\hline
a\cdot a = a, b\cdot b=b, c\cdot c=c,
a\cdot b = ab, a\cdot c=ac, b\cdot c=bc, \\ \hline
a\cdot bc = a(bc), b\cdot ac = b(ac), c\cdot ab=(\eta+1)(ab+bc+ac)-\eta(a+b+c)-a(bc)-b(ac), \\ \hline
\begin{tabu}{@{}c@{}}
a\cdot ab = \frac{1}{2}((1-\eta)a-\eta b + (1+2\eta)ab),
a\cdot ac = \frac{1}{2}((1-\eta)a-\eta c + (1+2\eta)ac), \\
b\cdot ab = \frac{1}{2}((1-\eta)b-\eta a + (1+2\eta)ab),
b\cdot bc = \frac{1}{2}((1-\eta)b-\eta c + (1+2\eta)bc), \\
c\cdot ac = \frac{1}{2}((1-\eta)c-\eta a + (1+2\eta)ac),
c\cdot bc = \frac{1}{2}((1-\eta)c-\eta b + (1+2\eta)bc),
\end{tabu}\\ \hline
a\cdot a(bc)=\frac{1}{2}((1-\eta)a-\eta bc + (1+2\eta)a(bc)),
b\cdot b(ac)=\frac{1}{2}((1-\eta)b-\eta ac + (1+2\eta)b(ac)), \\ \hline
\begin{tabu}{@{}c@{}}
b\cdot a(bc)=\frac{1}{4}(-\eta a+ (1-2\eta^2)b+ (\eta-2\eta^2)c+2\eta ab+ (4\eta^2-2\eta)bc-ac+2a(bc)+2b(ac)), \\
c\cdot a(bc)=\frac{1}{4}(-3\eta a-(2\eta^2+2\eta-1)c+ (2\eta+1)(ab+2ac-\eta b)+(4\eta^2+2)bc-2b(ac) ), \\
a\cdot b(ac)=\frac{1}{4}( (1-2\eta^2)a-\eta b+(\eta-2\eta^2)c+2\eta ab -bc +(4\eta^2-2\eta)ac+2a(bc)+2b(ac)), \\
c\cdot b(ac)=\frac{1}{4}(-3\eta b-(2\eta^2+2\eta-1)c+ (2\eta+1)(ab+2bc-\eta a)+(4\eta^2+2)ac-2a(bc) ),
\end{tabu} \\ \hline
\begin{tabu}{@{}c@{}}
(ab)^2=\frac{1}{4}((1-4\eta)(a+b)+(8\eta+2)ab),
(bc)^2=\frac{1}{4}((1-4\eta)(b+c)+(8\eta+2)bc), \\
(ac)^2=\frac{1}{4}((1-4\eta)(a+c)+(8\eta+2)ac), \\
ab\cdot bc=\frac{1}{4}( -3\eta(a+c)+(1-4\eta)b+(2\eta+1)(2ab+2bc+ac)-2b(ac)), \\
ab\cdot ac=\frac{1}{4}((1-4\eta)a-3\eta(b+c)+(2\eta+1)(2ab+bc+2ac)-2a(bc)), \\
bc\cdot ac=\frac{1}{4}(-\eta(a+b)+(1-2\eta)c-ab+2\eta(bc+ac)+ 2a(bc)+2b(ac)),
\end{tabu} \\ \hline
\begin{tabu}{@{}c@{}}
ab\cdot a(bc)=\frac{1}{8}( (2-8\eta)a-(2\eta^2+5\eta-1)b-(2\eta^2+\eta)c+ (10\eta+2)ab\\+(4\eta^2-2\eta+1)bc+2\eta ac+(4\eta+2)a(bc)), \\
bc\cdot a(bc)=\frac{1}{8}(-4\eta a-(4\eta^2+3\eta-1)(b+c)+(4\eta-1)(ab+ac)\\+(8\eta^2+2\eta+2)bc+6a(bc)), \\
ac\cdot a(bc)=\frac{1}{8}( (2-8\eta)a-(2\eta^2+\eta)b-(2\eta^2+5\eta-1)c + 2\eta ab\\+(4\eta^2-2\eta+1)bc+(10\eta+2)ac+(4\eta+2)a(bc)), \\
ab\cdot b(ac)=\frac{1}{8}((1-5\eta-2\eta^2)a+(2-8\eta)b-(2\eta^2+\eta)c+ (10\eta+2)ab\\+2\eta bc+(4\eta^2-2\eta+1)ac+(4\eta+2)b(ac)), \\
bc\cdot b(ac)=\frac{1}{8}(-(2\eta^2+\eta)a+(2-8\eta)b+(1-5\eta-2\eta^2)c+2\eta ab\\+(10\eta+2)bc+(4\eta^2-2\eta+1)ac+(4\eta+2)b(ac)), \\
ac\cdot b(ac)=\frac{1}{8}((1-3\eta-4\eta^2)(a+c)-4\eta b+ (4\eta-1)(ab+bc)\\+(8\eta^2+2\eta+2)ac+6b(ac)),
\end{tabu} \\ \hline
\begin{tabu}{@{}c@{}}
(a(bc))^2=\frac{1}{16}( (4-16\eta)a+(1-2\eta-8\eta^2)(b+c)+(8\eta-2)(ab+ac+(2\eta-1)bc)\\+(16\eta+12)a(bc) ), \\
(b(ac))^2=\frac{1}{16}( (1-2\eta-8\eta^2)(a+c)+(4-16\eta)b+ (8\eta-2)(ab+bc+(2\eta-1)ac)\\+(16\eta+12)b(ac) ), \\
a(bc)\cdot b(ac)=\frac{1}{16}( (2-8\eta-6\eta^2)(a+b)+(1-12\eta^2)c + (12\eta-3)ab\\+(12\eta^2-2\eta-1)(bc+ac)+(4\eta+8)(a(bc)+b(ac)) ).
\end{tabu} \\ \hline
\end{tabu}$}
\caption{Table of products for $\eta\neq-1$}\label{t:prod2}
\endgroup
\end{center}
\end{table}

In what follows, we assume that $A$ is generated by primitive axes $a$, $b$, and $c$.  Denote by $\mathcal{B}$ the
set $\{a, b, c, ab, ac, bc, a(bc), b(ac)\}$.
We also use the following notation for values of the
form: $\alpha=(a, b)$, $\beta=(b, c)$, $\gamma=(a, c)$, and $\psi= (a,bc)=(b,ac)=(c,ab)$.

\begin{prop}\label{p:3gen} The algebra $A$ is the span of $\mathcal{B}$.
The pairwise products of elements from $\mathcal{B}$ can be found according to Table~\ref{t:prod} and Table~\ref{t:prod2} depending on $\eta$. Moreover, if $\eta=-1$, then the Gram matrix for elements of $\mathcal{B}$ is as in Table~\ref{t:gram} and its determinant is equal to
\begin{multline*}
\frac{3}{2^9}(\alpha+\beta+\gamma-2\psi-1)^4\\\times(12\alpha\beta\gamma-2\alpha^2-2\beta^2-2\gamma^2-2\psi^2+2\alpha\beta+2\beta\gamma+2\alpha\gamma-4\alpha\psi-4\beta\psi-4\gamma\psi+\alpha+\beta+\gamma-2\psi+1)^3.
\end{multline*}

\end{prop}
\begin{proof}
We have several obvious products: $aa,bb,cc,ab,bc,ac, a(bc), b(ac)\in\mathcal{B}$. By Corollary~\ref{c:3prod},
$c(ab)$ lies in the span of $\mathcal{B}$.
Products $(ab)(ac)$, $(ab)^2$, $(ab)(a(bc))$, $(ab)(b(ac))$, and $(a(bc))^2$
can be found using smaller products with the aid of Lemmas~\ref{l:axay1} and~\ref{l:axay2}. We skip calculations and list final results in Tables~\ref{t:prod} and~\ref{t:prod2}.

Now we show how to find products $b(a(bc))$, $c(a(bc))$,
$a(b(ac))$, and $c(b(ac))$.
First consider $a(b(ac))$.
We use that $a(b(ac))=\frac{1}{2}(a(b(ac)+c(ab)))+\frac{1}{2}(a(b(ac)-c(ab)))$. If $\eta=-1$, then the first summand can be found as follows:
$$a(b(ac)+c(ab))=a(\beta a+\gamma b+\alpha a-a(bc))\in\langle\mathcal{B}\rangle.$$
Similarly, If $\eta\neq-1$, then we use the corresponding expression for $c(ab)$ from Corollary~\ref{c:3prod}.
It remains to find $a(b(ac)-c(ab))$.
Write $b=\alpha a+b_\eta+b_\frac{1}{2}$ and $c=\gamma a+c_\eta+c_\frac{1}{2}$. As in Lemma~\ref{l:3prod}, we see that
$b(ac)-c(ab)=x_\eta+x_\frac{1}{2}$,
where $x_\eta=(\alpha\eta^2-\alpha\eta)c_\eta-(\gamma\eta^2-\gamma\eta)b_\eta$ and
$x_\frac{1}{2}=\frac{\gamma}{4}x_\frac{1}{2}-\frac{\alpha}{4}y_\frac{1}{2}+(\frac{1}{2}-\eta)(b_\eta c_\frac{1}{2}-c_\eta b_\frac{1}{2})$. Clearly, $x_\eta\in A_\eta(a)$ and $x_\frac{1}{2}\in A_\frac{1}{2}(a)$. Then
$a(x_\eta+x_\frac{1}{2})=\eta x_\eta+\frac{1}{2}x_\frac{1}{2}=\frac{1}{2}(x_\eta+x_\frac{1}{2})+(\eta-\frac{1}{2})x_\eta$. Now $x_\eta+x_\frac{1}{2}=b(ac)-c(ab)$ and
$x_\eta$ is a linear combination of $b_\eta$ and $c_\eta$. As in Proposition~\ref{p:2gen},
we see that $$b_{\eta} = \frac{1}{1-2\eta}(\alpha a + b - 2ab), c_{\eta} = \frac{1}{1-2\eta}(\gamma a + c - 2ac).$$
So $a(b(ac)-c(ab))$ belongs to the span of $\mathcal{B}$ and the same is true for $a(b(ac))$.
Similar arguments can be applied to find $b(a(bc))$, $c(a(bc))$, and $c(b(ac))$.

Finally, we need to find the product of $a(bc)$ and $b(ac)$. If $\eta=-1$, then, by Corollary~\ref{c:3prod}, it is true that
$a(bc)+b(ac)=\beta a+\gamma b+\alpha c-c(ab)$, so
$$2(a(bc))(b(ac))=(\beta a+\gamma b+\alpha c)^2-2(\beta a+\gamma b+\alpha c)(c(ab))+(c(ab))^2-(a(bc))^2-(b(ac))^2.$$
The right-hand side of this equality is in $\mathcal{B}$ and can be found explicitly using the previous products. We write the final result into Table~\ref{t:1}. If $\eta\neq-1$, then we find the product in the same way using the expression for $c(ab)$ from Corollary~\ref{c:3prod}.
All calculations can be reproduced in the computer algebra system GAP~\cite{GAP}: we upload all commands for products and the Gram matrix in~\cite{file}.
\end{proof}

Finally, we describe the subspaces $A_\eta(a)$ and $A_\frac{1}{2}(a)$
to complete the proof of Theorem~\ref{th:2}.
\begin{lem}\label{l:eigenspaces}
The following statements hold.
\begin{enumerate}[(i)]
\item $A_{\eta}(a)=\langle \alpha a+b-2ab, \gamma a+c-2ac, \psi a+bc-2a(bc)\rangle$;
\item $
A_\frac{1}{2}(a)=\langle \alpha(\eta-1)a-\eta b+ab, \gamma(\eta-1)a-\eta c+ac, \psi(\eta-1)a-\eta bc+a(bc)$,\\
$(2\eta^2-1)(2\psi-\beta)a-\eta\gamma b+(\eta-2\eta^2)\alpha c-bc+2b(ac)\rangle.$
\end{enumerate}
\end{lem}
\begin{proof}
Write $b=\alpha a+b_\eta+b_\frac{1}{2}$, where $b_\eta\in A_\eta(a)$ and $b_\frac{1}{2}\in A_\frac{1}{2}(a)$. As in Proposition~\ref{p:2gen},
we see that $\alpha a+b-2ab=(1-2\eta)b_\eta\in A_\eta(a)$ and
$\alpha(\eta-1)a-\eta b+ab=(1-2\eta)b_\frac{1}{2}\in A_\frac{1}{2}(a)$. Similarly, we see that $\gamma a+c-2ac, \psi a+bc-2a(bc)\in A_\eta(a)$ and
$\gamma(\eta-1)a-\eta c+ac, \psi(\eta-1)a-\eta bc+a(bc)\in A_\frac{1}{2}(a)$.
It is straightforward calculation with the aid of Tables~\ref{t:prod} and~\ref{t:prod2} that
$(2\eta^2-1)(2\psi-\beta)a-\eta\gamma b+(\eta-2\eta^2)\alpha c-bc+2b(ac)\in A_\frac{1}{2}(a)$.

Now write coefficients of these seven vectors together with $a$, which spans $A_1(a)$, with respect to $a,b,c,ab,bc,ac,a(bc),b(ac)$ into $8\times 8$ matrix. Using GAP,
we find that the determinant of this matrix equals
$16(\eta-\frac{1}{2})^3$. Then Proposition~\ref{p:3gen} implies that these eight eigenvectors of $ad_a$ span $A$ and hence the corresponding vectors span $A_\eta(a)$ and $A_\frac{1}{2}(a)$.
\end{proof}

We conclude this section with the following statement which is a consequence of Proposition~\ref{p:3gen} and Lemma~\ref{l:eigenspaces}.
\begin{cor}\label{cor:minus-one}
Fix a field $\mathbb{F}$ of characteristic not $2$ or $3$. Take arbitrary $\eta\in\mathbb{F}\setminus\{1,\frac{1}{2}\} $. If $\eta=-1$, then choose
arbitrary values of the parameters $\alpha,\beta,\gamma,\psi\in\mathbb{F}$. If $\eta\neq-1$, then set $\alpha=\beta=\gamma=\psi=1$.
Assume that $e_1, e_2,\ldots,e_8$ is a basis of an
$8$-dimensional vector space $V$ over $\mathbb{F}$. Define products on these elements as in Tables~\ref{t:prod} and~\ref{t:prod2}, where
we identify elements of the basis with $a$, $b$, $c$, $ab$, $bc$, $ac$, $a(cb)$, $b(ac)$, $c(ab)$ in these tables. Then $V$ with the thus-defined product is a $\mathcal{PC}(\eta)$-axial algebra generated
by the three primitive axes $e_1$, $e_2$, and $e_3$.

Moreover, suppose that $A$ is a $\mathcal{PC}(\eta)$-axial algebra
endowed with the Frobenius form as in Proposition~\ref{p:form} and generated by three primitive axes $x$, $y$, and $z$.
If $(x, y)=\alpha$, $(y, z)=\beta$, $(x,z)=\gamma$, and $\psi=(x, yz)$, then $A$ is a homomorphic image of $V$.
\end{cor}

\Addresses


\begin{thebibliography}{1}

\bibitem{W99} S.~Walcher, On algebras of rank three,
{\it Comm. Algebra}, {\bf27}:7 (1999), 3401--3438.

\bibitem{MO93} K.~Meyberg, J.M.~Osborn, Pseudo-composition algebras, {\it Math Z.}, {\bf214}:1 (1993), 67--77.

\bibitem{EO00}
A.~Elduque, S.~Okubo, On algebras satisfying $x^2x^2=N(x)x$, {\it Math.~Z.}, {\bf235}:2 (2000), 275--314.

\bibitem{S59} T.~A. Springer, On a class of Jordan algebras, {\it Indag. Math.}, {\bf21} (1959), 259--264.

\bibitem{Eth39}
I.M.H.~Etherington, Genetic algebras,
{\it Proc. Roy. Soc. Edinburgh}, {\bf59} (1939), 242--258.

\bibitem{HRS1} J.I.~Hall, F.~Rehren and S.~Shpectorov, Universal axial algebras and a theorem of Sakuma, {\it J. Algebra}, {\bf421} (2015), 394--424.

\bibitem{HRS2} J.I.~Hall, F.~Rehren, S.~Shpectorov, Primitive axial algebras of Jordan type, {\it J. Algebra}, {\bf437} (2015), 79--115.

\bibitem{Gr82}
R.L.~Griess Jr., The friendly giant, {\it Ivent. Math.}, {\bf 69} (1982), 1--102.

\bibitem{Iv09} A.A.~Ivanov, The Monster Group and Majorana Involutions, Cambridge Tracts in Mathematics 176,
Cambridge University Press, (2009), 266 pp.

\bibitem{survey} J.~McInroy, S.~Shpectorov, {\it Axial algebras of Jordan and Monster type}, \href{https://arxiv.org/abs/2209.08043}{arXiv:2209.08043}.

\bibitem{Wal88} H.R\"{o}hrl, S.~Walcher, Algebras of complexity one, {\it Algebras Groups Geom.}, {\bf 5}:1 (1988),
61--107.

\bibitem{Whybrow2} M.~Whybrow, Graded 2-generated axial algebras, \href{https://arxiv.org/abs/2005.03577}{arXiv:2005.03577}.

\bibitem{Wal88-2}
S.~Walcher, Bernstein algebras which are Jordan algebras,
{\it Arch. Math. (Basel)}, {\bf 50}:3 (1988), 218--222.

\bibitem{hss} J.I.~Hall, Y.~Segev, S.~Shpectorov, On primitive axial algebras of Jordan type,
Bulletin of the Institute of Mathematics, {\bf 13}:4, 397--409.

\bibitem{hss2}
J.I.~Hall, Y.~Segev and S.~Shpectorov, Miyamoto involutions in axial algebras of
Jordan type half, {\it Israel J. Math.} {\bf223}:1 (2018), 261--308.

\bibitem{GAP}
The GAP Group, GAP -- Groups, Algorithms, and Programming, Version 4.12.2, 2022.
(\url{https://www.gap-system.org})

\bibitem{file}
\url{https://github.com/AlexeyStaroletov/AxialAlgebras/tree/master/PC-eta}

\end{thebibliography}
\end{document}